\documentclass[reqno,a4paper,12pt]{amsart}
\usepackage{amsmath,amscd,amsfonts,amssymb}
\usepackage{mathrsfs,dsfont}
\numberwithin{equation}{section}

\usepackage{color}
\usepackage[normalem]{ulem}
\usepackage[unicode=true,colorlinks=true,linkcolor=blue,citecolor=blue,pagebackref,hypertexnames=true]{hyperref}

\def\R{\mathbb{R}}

\def\Z{\mathbb{Z}}

\def\T{\mathbb{T}}

\def\Lam{\Lambda}
\def\1{\mathds{1}}
\def\eps{\varepsilon}

\renewcommand\le{\leqslant}
\renewcommand\ge{\geqslant}
\renewcommand\leq{\leqslant}

\renewcommand\hat{\widehat}

\newcommand{\ft}[1]{\widehat #1}

\newcommand\mes{\operatorname{mes}}

\newcommand{\supp}{\operatorname{supp}}

\theoremstyle{plain}
\newtheorem{theorem}{Theorem}[section]

\newtheorem{lemma}{Lemma}[section]

\newtheorem*{claim*} {Claim}

\newcommand{\thmref}[1]{Theorem~\ref{#1}}

\newcommand{\lemref}[1]{Lemma~\ref{#1}}

\theoremstyle{definition}
\newtheorem*{remark*}{Remark}

\newenvironment{enumerate-math}
{\begin{enumerate}
\addtolength{\itemsep}{5pt}
}
{\end{enumerate}}

\begin{document}

 \title{On non-periodic tilings of the real line by a function}

\author[M. Kolountzakis]{{Mihail N. Kolountzakis}}
\address{M.K.: Department of Mathematics and Applied Mathematics, University of Crete, Voutes Campus, GR-700 13, Heraklion, Crete, Greece}
\email{kolount@uoc.gr}

\author[N. Lev]{{Nir Lev}}
\address{N.L.: Department of Mathematics, Bar-Ilan University, Ramat-Gan 52900, Israel}
\email{levnir@math.biu.ac.il}

\thanks{M.K.\ has been partially supported by the ``Aristeia II'' action (Project
FOURIERDIG) of the operational program Education and Lifelong Learning
and is co-funded by the European Social Fund and Greek national resources.}

\thanks{N.L.\ is partially supported by the Israel Science Foundation grant No. 225/13.}

\begin{abstract}
It is known that a positive, compactly supported function $f \in L^1(\mathbb R)$ 
can tile by translations only if the translation set is a finite union of periodic sets.
We prove that this is not the case if $f$ is allowed to have unbounded support.
On the other hand we also show that if the  translation set has finite local complexity,  then it must be periodic,
even if the support of $f$ is unbounded.
\end{abstract}

\maketitle

\section{Introduction}

\subsection{} Let $f$ be a function in $L^1(\mathbb R)$ and $\Lambda$ be a discrete set in $\mathbb R.$
The function $f$ is said to tile $\R$ at level $w$ (a constant) with translation set $\Lambda$ if
\begin{equation}\label{1.1}
\sum_{\lambda\in\Lambda}f(x-\lambda)=w\quad\text{(a.e.)}
\end{equation}
 and the series \eqref{1.1} converges absolutely a.e. In this case we will say
that $f+\Lambda$ is a tiling of $\R$ at level $w$.

 It was proved in \cite{LaWa96} that if $f=\1_\Omega$ is the indicator function of a bounded set $\Omega \subset \R$, whose boundary has Lebesgue measure zero, and if $f+\Lambda$ is a tiling at level $1,$ then $\Lambda$ must be a periodic set, that is, $\Lambda+\tau=\Lambda$ for some $\tau>0$.

 This result was extended in \cite{KoLa96} (and proved earlier in \cite{LepHo91}) to tilings by a function $f\in L^1(\mathbb R)$ with compact support. Namely, it was proved that if such a function $f$ tiles at some level $w$ with translation set $\Lambda$ of bounded density (and $f$ is not identically zero), then $\Lambda$ must be a finite union of periodic sets.

 Recall that a set $\Lambda\subset\mathbb R$ is said to be of bounded density if
 \begin{equation}\label{1.2}
 \sup_{x\in\mathbb R} \#(\Lambda\cap[x,x+1))<\infty.
 \end{equation}
 It was shown in \cite[Lemma 2.1]{KoLa96} that any tiling by a non-negative function $f$ (not identically zero) is necessarily of bounded density.

 \subsection{}
The question whether the periodic structure of the tiling is still necessary even when the function $f$ is allowed to have unbounded support, has remained open.  In this paper  we answer this question in the negative:

 \begin{theorem}\label{thm1.1}
 There is a positive function $f\in L^1(\mathbb R)$ and a discrete set $\Lambda$ of bounded density such that $f+\Lambda$ is a tiling of $\R$
at level $1$, but $\Lambda$ is not a finite union of periodic sets.
 \end{theorem}

 Actually, the set $\Lambda$ in our example is an arbitrarily small perturbation of the set of integers $\mathbb Z,$ and $f$ is a function from the Schwartz class.

 The proof of \thmref{thm1.1} depends on a result due to Kargaev \cite{Kar82} who constructed a set $\Omega\subset \mathbb R$ of finite Lebesgue measure which has a spectral gap. The latter means that the Fourier transform of the indicator function $\1_\Omega$ vanishes on some interval.

 \subsection{}
 A discrete set $\Lambda\subset\mathbb R$ is said to have \emph{finite local complexity} if $\Lambda$ can be enumerated as a sequence $\{\lambda(n)\},$ $n\in\mathbb Z,$ such that $\lambda(n)<\lambda(n+1)$ and the successive differences $\lambda(n+1)-\lambda(n)$ take only finitely many different values.

Clearly, any periodic set is of finite local complexity, but the converse is not true.

The next result complements \thmref{thm1.1}:

\begin{theorem}\label{thm1.2}
Let $f\in L^1(\mathbb R)$ (possibly of unbounded support) and $\Lambda$ be a set of finite local complexity.
If $f$ is not identically zero and $f+\Lambda$ is a tiling of $\R$ at some level $w$, then $\Lambda$ must be a periodic set.
\end{theorem}

\subsection{} One may also consider tilings of the set of integers $\mathbb Z$ by a set, or, more generally, by a function. If $f$ is a function defined on $\mathbb Z$ and $\Lambda$ is a subset of $\mathbb Z$, then we say that $f+\Lambda$ is a tiling of $\Z$ at level $w$ if
\begin{equation}\label{6.1}
\sum_{k\in\Lambda}f(n-k)=w,\quad n\in\mathbb Z,
\end{equation}
and the series \eqref{6.1} converges absolutely for each $n.$

It is known, see e.g.\ \cite{New77}, that any tiling of the integers (at level 1) by translates of a finite set $\Omega \subset \Z$ must be periodic.
Here we extend this result to tilings by functions:

\begin{theorem}\label{thm6.1}
Let $f\in\ell^1(\mathbb Z)$ (possibly of unbounded support) and $\Lambda$ be a subset of $\mathbb Z$.
If $f$ is not identically zero and $f+\Lambda$ is a tiling of $\Z$ at some level $w$, then $\Lambda$ is periodic.
\end{theorem}

Actually this result can be deduced from \thmref{thm1.2}, but we will provide an alternative, independent proof of it.

\subsection{}
The rest of the paper is organized as follows.
In Section \ref{sec:kar} we give a self-contained presentation of a weaker version of
 Kargaev's theorem on sets with a spectral gap. In Section \ref{sec:nonp} we use this result to
prove the existence of a  non-periodic tiling of $\mathbb R$.

In Section \ref{sec:fourier} we give a Fourier-analytic condition necessary for tiling. In earlier work \cite{Ko00a,Ko00b,KoLa96} this condition was obtained only under additional assumptions on the tiling function $f$. In Section \ref{sec:flc} we use this condition to show that tilings of $\mathbb{R}$ by translation sets of finite local complexity, as well as general tilings of $\mathbb{Z}$, must be periodic.

In the last Section \ref{sec:open} we mention some open problems.


\section{Kargaev's construction}\label{sec:kar}

\subsection{}
Kargaev constructed in \cite{Kar82}  a (necessarily unbounded) set $\Omega\subset \mathbb R$ of finite Lebesgue measure 
which has a spectral gap, namely, such that the Fourier transform of its indicator function  vanishes on some interval.
Other presentations may be found in \cite[pp. 376--392]{HaJo94} and  \cite{KaVo92}. 

The paper \cite{Kar82} also contains a simpler proof of the following weaker result: there is a
function $F$ on the real line, which is $\{-1,0,1\}$-valued, and has a spectral gap. In this version, however,
the function $F$ is not in $L^1(\R)$, and its Fourier transform is understood in the sense of distributions.

In the above mentioned presentations, the construction relies on the implicit function theorem in infinite dimensional Banach space.

In this section we give a self-contained presentation of the weaker version of Kargaev's construction, in a
form that will be sufficient for our application. In the proof we avoid referring to the  infinite dimensional  implicit function 
theorem, and merely use the Banach contractive mapping theorem.

\subsection{}
 Let $\{\alpha(n)\},$ $n\in\mathbb Z,$ be a bounded sequence of real numbers. To such a sequence
 we associate a function $F$ on the real line, defined by
 \begin{equation}\label{2.1}
F(x)=\sum_{n\in\mathbb Z} F_n(x),\quad x\in\mathbb R,
 \end{equation}
 where $F_n=\1_{[n,n+\alpha(n)]}$ if $\alpha(n)\ge0$ and $F_n=-\1_{[n+\alpha(n),n]}$ if $\alpha(n)<0.$

 The assumption that the sequence $\{\alpha(n)\}$ is bounded implies that \eqref{2.1} represents a bounded function on $\mathbb R.$ In particular, $F$ is a tempered distribution.

In this paper, the Fourier transform is normalized as follows,
 $$\hat\varphi(t)=\int_{\mathbb R}\varphi(x)e^{-2\pi itx}dx,$$
 and this definition is extended to tempered distributions on $\R$ in the usual way.

 \begin{theorem}[Kargaev \cite{Kar82}]\label{thm2.1}
 Let $0<a<\frac{1}{2}$ and $\varepsilon>0$ be given. Then there is a real sequence $\{\alpha(n)\},$ $n\in\mathbb Z,$ satisfying
 \begin{equation}\label{2.2}
 0<\sup_{n\in\mathbb Z}|\alpha(n)|<\varepsilon,\qquad\lim_{n\to\pm\infty}\alpha(n)=0,
 \end{equation}
 and such that the function $F$ defined by \eqref{2.1} satisfies $\hat F=0$ in $(-a,a).$
 \end{theorem}

 In fact, the proof below gives a sequence $\{\alpha(n)\}$ such that $\sum\alpha(n)^2<\infty.$

 The proof  is divided into a series of lemmas.

\subsection{}
We need the following two simple inequalities.

 \begin{lemma}\label{lem2.1}
For any $\theta\in\mathbb R$ we have
\begin{enumerate-math}
\item $|e^{i\theta}-1|\le|\theta|;$
\item $|e^{i\theta}-1-i\theta|\le\frac{1}{2}|\theta|^2.$
\end{enumerate-math}
\end{lemma}

\begin{proof}
This follows from the identities
\[
e^{i\theta}-1 = i\theta\int_0^1 e^{i\theta t}dt, \qquad
e^{i\theta}-1-i\theta=(i\theta)^2\int_0^1 e^{i\theta t}(1-t)dt. \qedhere
\]
\end{proof}

\subsection{}
Denote $I:=\left[-\frac{1}{2},\frac{1}{2}\right].$ Let $X$  be the space of all continuous functions $f:I\to \mathbb C$ satisfying $f(-t)=\overline{f(t)},$ $t\in I.$ We regard $X$ as a real Banach space 
endowed with the  norm $\|f\|_\infty=\sup|f(t)|,$ $t\in I.$

Notice that the Fourier coefficients
$$\hat f(n)=\int_I f(t)e^{-2\pi int}dt,\quad n\in\mathbb Z,$$
of an element $f\in X$ satisfy
$$\hat f(n)\in\mathbb R\quad (n\in \mathbb Z),\qquad \sum_{n\in\mathbb Z}\hat f(n)^2\le \|f\|_\infty^2.$$

\subsection{}
Define a (non-linear) map $R:X\to X$ by

\begin{equation}\label{2.3}
(Rf)(t)=\sum_{n\in\mathbb Z} e^{2\pi int}\cdot \frac{e^{2\pi i\hat f(n)t}-1-2\pi i\hat f(n)t}{2\pi it} \, .
\end{equation}
Observe that the series converges uniformly on $I.$ Indeed, by \lemref{lem2.1}, the $n$'th member of the sum is bounded by $\frac{\pi}{2}\hat f(n)^2$ on $I,$ which implies the uniform convergence. As each member of the sum belongs to $X,$ the same is true for the entire sum $Rf.$

We also conclude that
\begin{equation}\label{2.4}
\|Rf\|_\infty\le\frac{\pi}{2}\sum_{n\in\mathbb Z}\hat f(n)^2\le \frac{\pi}{2}\|f\|_\infty^2
\end{equation}
for any $f\in X.$

\begin{lemma}\label{lem2.2}
For any $f,g\in X$ we have
$$\|Rf-Rg\|_\infty\le\frac{\pi}{2}\|f-g\|_\infty^2+\pi\|f\|_\infty \|f-g\|_\infty.$$
\end{lemma}

\begin{proof}
We have
\begin{align*}
(Rg-Rf)(t)&=\sum_{n\in\mathbb Z} e^{2\pi int} \cdot \frac{e^{2\pi i \, \hat g(n)t}-e^{2\pi i\hat f(n)t}-2\pi i \big(\hat g(n)-\hat f(n)\big)t}{2\pi it}\\
&=\sum_{n\in\mathbb Z}\Bigg[e^{2\pi i(n+\hat f(n))t}\cdot\frac{e^{2\pi i(\hat g(n)-\hat f(n))t}-1-2\pi i \big(\hat g(n)-\hat f(n)\big)t}{2\pi it}\\
&\quad +e^{2\pi int} \big(e^{2\pi i\hat f(n)t}-1\big)\big(\hat g(n)-\hat f(n)\big)\Bigg].
\end{align*}
Hence by \lemref{lem2.1},
\begin{align*}
&|(Rg-Rf)(t)| \le \sum_{n\in\mathbb Z}\left[\pi |t| \cdot \big(\hat g(n)-\hat f(n)\big)^2+2\pi |t| \cdot |\hat f(n)|\cdot|\hat g(n)-\hat f(n)|\right]\\
&\qquad \le \pi |t|  \sum_{n\in\mathbb Z}\big(\hat g(n)-\hat f(n)\big)^2+ 2\pi |t| \left(\sum_{n\in\mathbb Z}\hat f(n)^2\right)^{1/2} \left(\sum_{n\in\mathbb Z}\big(\hat g(n)-\hat f(n)\big)^2\right)^{1/2}.
\end{align*}
Since $|t| \leq 1/2$ on $I$ this implies the claim.
\end{proof}

\begin{remark*}
It follows from \lemref{lem2.2} that $R$ is a continuous mapping $X\to X.$
\end{remark*}

\begin{lemma}\label{lem2.3}
There exist absolute constants $c>0$ and $0<\rho<1$  such that
\begin{equation}\label{2.5}
\|Rf-Rg\|_\infty\le\rho \, \|f-g\|_\infty
\end{equation}
whenever $f,g\in X$ are such that $\|f\|_\infty\le c$, $\|g\|_\infty\le c$.
\end{lemma}

\begin{proof} It follows from \lemref{lem2.2} that 
$$\|Rf-Rg\|_\infty\le 2\pi c \, \|f-g\|_\infty$$
whenever $\|f\|_\infty\le c,$ $\|g\|_\infty\le c$. Hence it is enough to choose $\rho := 2\pi c < 1$.
\end{proof}

\begin{lemma}\label{lem2.4}
 If $\eps>0$ is sufficiently small then for any $g\in X,$  $\|g\|_\infty\le  \eps,$  there exists
an element $f\in X$ such that $f+Rf=g$ and  $\|f\|_\infty \le 2\eps$.
\end{lemma}

 \begin{proof} Fix $g\in X$ such that $\|g\|_\infty \le \eps.$ Define a map $H: X\to X$ by
 $$H f :=g-Rf.$$
 Notice that $f\in X$ satisfies $f+Rf=g$ if and only if $f$ is a fixed point of the map $H.$
 
 Denote $B:=\{f\in X:\|f-g\|_\infty \le \eps\}.$ Let us show that $H(B)\subset B.$
 Indeed, if $f\in B$ then $\|f\|_\infty \le 2\eps,$ hence using \eqref{2.4} we have
 $$\|H(f)-g\|_\infty=\|Rf\|_\infty \le \frac{\pi}{2}  \cdot (2\eps)^2<\eps,$$
provided that $2\pi \eps < 1$. That is, $H(B)\subset B$.
 
It now follows  from \lemref{lem2.3} that if $2\eps < c$, then $H$ is a contractive mapping acting on the closed subset $B$ of $X.$ 
Hence by the Banach contractive mapping theorem,
there is a (unique) fixed point $f\in B$ of $H,$ which yields the required solution.
 \end{proof}
 
\subsection{}
Now we can finish the proof of \thmref{thm2.1}. 

\begin{proof}[Proof of \thmref{thm2.1}]
Choose a function $g\in X,$ $\|g\|_\infty<\varepsilon/2$, such that $g$ vanishes on $(-a,a)$ but  does not vanish identically on $I.$ Use \lemref{lem2.4} to find $f\in X$ such that
 $$f+Rf=g,\qquad \|f\|_\infty<\varepsilon.$$
 Set $$\alpha(n):=\hat f(n), \quad n\in\mathbb Z,$$
 and let $F$ be the function associated to this sequence $\{\alpha(n)\}$ defined by \eqref{2.1}. We have
 $$\hat F=\lim_{N\to\infty}\sum_{|n|\le N}\hat F_n$$
 in the sense of distributions, and
 $$\hat F_n(-t)=\frac{e^{2\pi i\alpha(n)t}-1}{2\pi it} \cdot e^{2\pi int} .$$
Hence
\begin{equation}\label{2.6}
\hat F(-t) =
\lim_{N\to\infty} \Big[ \sum_{|n|\le N} \hat f(n)e^{2\pi int} + \sum_{|n|\le N}
\frac{e^{2\pi i\hat f(n)t}-1-2\pi i\hat f(n)t}{2\pi it} \cdot e^{2\pi int} \Big].
\end{equation}

The first sum converges in $L^2(I)$ to the function $f$, while the second sum converges uniformly on $I$
to $Rf$. It follows that the distribution $\hat F$ satisfies
$$\hat F(-t)=f(t)+(Rf)(t)=g(t)$$
in the open interval $\left(-\frac{1}{2},\frac{1}{2}\right)$.
Hence, as $g$ vanishes in $(-a,a)$, the same is true for $\hat F.$

Finally,  notice that $f\ne0$ since $g\ne 0$, and thus
$$0<\sup_{n\in\mathbb Z}|\alpha(n)|\le \|f\|_\infty<\varepsilon.$$
Moreover, $\alpha(n)$ tends to zero as $n \to \pm \infty$ (in fact, $\sum \alpha(n)^2<\infty)$.
This completes the proof of \thmref{thm2.1}.
\end{proof}


\section{Construction of a non-periodic tiling}
\label{sec:nonp}

Here we use \thmref{2.1} to prove \thmref{thm1.1}.

\begin{proof}[Proof of \thmref{thm1.1}]

Let $\{\alpha(n)\}$ be the sequence given by \thmref{thm2.1}, and define
\begin{equation}\label{3.1}
\Lambda=\{n+\alpha(n)\},\quad n\in\mathbb Z.
\end{equation}
Observe that the distributional derivative of the function $F$ in \eqref{2.1} is
$$F'=\sum_{n\in\mathbb Z}(\delta_n-\delta_{n+\alpha(n)}) = \delta_{\mathbb Z} - \delta_\Lambda,$$
where we denote
\begin{equation*}
\delta_{\mathbb Z} =\sum_{n\in\mathbb{Z}}\delta_n, \qquad
\delta_\Lambda=\sum_{n\in\mathbb{Z}} \delta_{n+\alpha(n)}.
\end{equation*}
By Poisson's summation formula we have $\hat \delta_{\mathbb Z}=\delta_{\mathbb Z}.$
Hence, it follows that
$$\hat \delta_\Lambda=\delta_{\mathbb Z}-\ft{{F'}}.$$
Since $(-a,a)$ is a spectral gap for $F,$ the same is true for $F'.$ We deduce that
$$\hat \delta_\Lambda=\delta_0\quad\text{in $(-a,a)$.}$$

Let $f$ be a positive Schwartz function, with integral one, whose Fourier transform $\hat f$ has compact
support contained in $(-a,a)$. Denote $\varphi_x(t)=\hat f(t)\exp 2\pi ixt,$ then 
$$1=\langle\delta_0,\varphi_x\rangle=\langle\hat \delta_\Lambda,\varphi_x\rangle=\langle \delta_\Lambda,\hat\varphi_x\rangle=\sum_{\lambda\in\Lambda}f(x-\lambda).$$
Hence $f+\Lambda$ is a tiling at level 1.

Finally, observe that $\Lambda$ is not the finite union of  periodic sets. Indeed, since $\alpha(n)\to0$ as $n\to\pm\infty,$ any periodic set contained in $\Lambda$ must also be contained in $\mathbb Z.$ But $\Lambda$ itself is not contained in $\mathbb Z,$ since the $\alpha(n)$ are not all zero.
This completes the proof of \thmref{thm1.1}.
\end{proof}

\begin{remark*}
The fact that the measure $\delta_\Lambda$ defined above
admits a spectral gap was already pointed out by Kargaev, see \cite[Section 4.2]{Kar82}.
\end{remark*}


\section{Fourier analytic condition for tiling}\label{sec:fourier}

For a discrete set $\Lambda$ in $\mathbb{R}$ we define the measure
\begin{equation}\label{4.3}
\delta_\Lambda=\sum_{\lambda\in\Lambda}\delta_\lambda.
\end{equation}
If $\Lambda$ has bounded density (or, more generally, if  $\#(\Lambda\cap(-R,R))$ increases polynomially in $R$)
then the measure $\delta_\Lambda$ is a tempered distribution on $\mathbb R$.

\begin{theorem}\label{thm4.1}
Let $f\in L^1(\mathbb R)$ and $\Lambda$ be a discrete set of bounded density in $\mathbb R.$ If $f+\Lambda$ is a tiling
at some level $w$, then
\begin{equation}\label{4.1}
\supp(\hat\delta_\Lambda)\subset \{\hat f=0\}\cup\{0\}.
\end{equation}
\end{theorem}

This result was proved in \cite{KoLa96} under the extra assumption that the Fourier transform $\hat f$ is a smooth function.
In \cite{Ko00a,Ko00b} another version of the result was proved, where instead of smoothness it was assumed that
$\hat f$ is non-negative and has compact support. Our goal in \thmref{thm4.1} is to remove these additional assumptions.

\begin{proof}
Let $t_0$ be a point such that $\hat f(t_0)\ne0,$ and $t_0\ne0.$ Then $\hat f$ has no zeros 
in some open interval $J$ containing $t_0,$ but such that $J$ does not contain $0.$ 
We will show that $\hat\delta_\Lambda$ vanishes on $J,$ which implies that $t_0$ lies
 outside of the support of $\hat\delta_\Lambda.$

Let therefore $\psi$ be an infinitely smooth function, whose support lies in $J.$ We must show that $\langle \hat \delta_\Lambda,\psi\rangle=0.$
Let $[a,b]\subset J$ be a closed interval which contains $\supp(\psi).$ Then there is $g\in L^1(\mathbb R)$ such that $\hat f\cdot\hat g=1$ on $[a,b]$ (this is essentially Wiener's Tauberian theorem, see e.g.\ \cite[Section~6.2]{Hel10}).

The function $\psi$ is the Fourier transform of some function $\varphi$ in the Schwartz class. We have
$$\hat\varphi\cdot \hat g\cdot\hat f=\hat\varphi,$$
and hence
\begin{equation}\label{4.4}
\varphi\ast g\ast f=\varphi.
\end{equation}
It follows that
\begin{align}
\langle\hat\delta_\Lambda,\psi\rangle&=\langle\delta_\Lambda,\hat\psi\rangle=\sum_{\lambda\in
\Lambda}\varphi(-\lambda)=\sum_{\lambda\in\Lambda}(\varphi\ast g\ast f)(-\lambda)\nonumber\\
&=\sum_{\lambda\in\Lambda}\int_{\mathbb R}(\varphi\ast g)(-x)f(x-\lambda)dx.\label{4.5}
\end{align}

Now we need the following

\begin{claim*}
We have
\begin{equation}\label{4.6}
\sum_{\lambda\in\Lambda}\int_{\mathbb R}|(\varphi\ast g)(-x)|\cdot|f(x-\lambda)|dx<\infty.
\end{equation}
\end{claim*}

The claim allows us to exchange the sum and integral in \eqref{4.5}. Assuming that $f+\Lambda$ is a tiling at level $w,$ we get
$$\langle\hat\delta_\Lambda,\psi\rangle=\int_{\mathbb R}(\varphi\ast g)(-x)\sum_{\lambda\in\Lambda}f(x-\lambda)dx=w\int_{\mathbb R}(\varphi\ast g)(-x)dx=w\cdot\hat\varphi(0)\cdot\hat g(0).$$
But since $\hat\varphi=\psi$ and $0\notin \supp(\psi)$, it follows that $\langle\hat\delta_\Lambda,\psi\rangle=0$, as needed.

It remains to prove the claim. Indeed, the left hand side of \eqref{4.6} is not greater than
\begin{equation}\label{4.7}
\sum_{\lambda\in\Lambda}\int_{\mathbb R}(|\varphi|\ast|g|)(-x)\cdot|f(x-\lambda)|dx=\int_{\mathbb R}(|f|\ast|g|)(-x)\sum_{\lambda\in\Lambda}|\varphi(x-\lambda)|dx.
\end{equation}
The inner sum on the right hand side of \eqref{4.7} is a bounded function of $x,$ since $\varphi$ is a Schwartz function and $\Lambda$ has bounded density, while $|f|\ast|g|$ is a function in $L^1(\mathbb R).$ Hence the integral in \eqref{4.7} converges, and this completes the proof.
\end{proof}

\begin{remark*}
The proof above also shows (by choosing $t_0=0$) that if $\int f\ne 0$, then there is $a>0$ such that
$$\hat\delta_\Lambda=c\cdot\delta_0\quad\text{in $(-a,a)$,}$$
where $c=w\cdot(\int f)^{-1}.$
\end{remark*}


\section{Tilings of finite local complexity are periodic}
\label{sec:flc}

\subsection{}
The following result was proved in \cite{IoKo13}, although it was not explicitly  stated there  in this form.

 \begin{theorem} \label{thm5.1}
Let $\Lambda$ be a set of finite local complexity in $\mathbb{R}$. If the distribution $\hat \delta_\Lambda$ vanishes 
on some open interval $(a,b)$ then $\Lambda$ must be a periodic set.
 \end{theorem}

The result stated explicitly in \cite{IoKo13} was that if  $\Omega$ is a bounded set in $\R$, $\mes(\Omega)=1$, and if  the function $f = |\ft{\1}_\Omega|^2$ tiles $\R$ at level $1$ with some translation set $\Lambda$, then $\Lambda$ is necessarily periodic.  However, the periodicity of $\Lam$ was deduced using only the fact that such $\Lam$ must have finite local complexity, and $\hat \delta_\Lambda$ must vanish on some open interval, see  \cite[Sections 2.3--2.4]{IoKo13}.

Using Theorems \ref{thm4.1}  and \ref{thm5.1}  we can now prove \thmref{thm1.2}.

\begin{proof}[Proof of \thmref{thm1.2}]

Assume that $f+\Lambda$ is a tiling at some level $w$, where $f\in L^1(\mathbb R)$ and $\Lambda$ is a set of finite local complexity. In particular, $\Lambda$ has bounded density. Hence by \thmref{thm4.1} we have
\begin{equation}\label{5.1}
\supp(\hat\delta_\Lambda)\subset\{\hat f=0\}\cup\{0\}.
\end{equation}
Since $f$ is not identically zero, the closed set on the right hand side of \eqref{5.1} is not the whole $\mathbb R. $ So there is an open interval $(a,b)$ disjoint from $\supp(\hat \delta_\Lambda),$ that is, the measure $\delta_\Lambda$ has a spectral gap.
Thus, by \thmref{thm5.1}, $\Lambda$ must be a periodic set. 
\end{proof}

\subsection{}
It remains to prove \thmref{thm6.1}. As we have mentioned, it can be deduced from \thmref{thm1.2}, but we will provide an alternative, independent proof. The point is that when tilings of $\Z$ are considered, then instead of using Theorems \ref{thm4.1}  and \ref{thm5.1}, the proof can be based on some classical results in Fourier analysis.

Observe that if  $f\in\ell^1(\mathbb Z)$ and $\Lambda \subset \mathbb Z$, then the condition that $f+\Lambda$ is a tiling of $\Z$ at level $w$ means that
\begin{equation}\label{5.2a}
(f \ast \1_\Lambda)(n) = w, \quad n \in \mathbb Z,
\end{equation}
namely, the convolution of $f$ with the indicator function $\1_\Lam$ is the constant function which is equal to $w$ on $\Z$. This  can be reformulated by saying that
\begin{equation}\label{5.2b}
\ft{f} \cdot \ft{\1}_\Lam = w \cdot \delta_0,
\end{equation}
where $\ft f$ is understood as an element of the Wiener algebra $A(\T)$  of  continuous functions on the circle $\mathbb T = \mathbb R / \mathbb Z$
with absolutely convergent Fourier series, and $\ft{\1}_\Lam$ is understood as a ``pseudo-measure'' on $\T$, that is, as a distribution with bounded Fourier coefficients.

The following result is basically due to Wiener, but its formulation in terms of  pseudo-measures is due to Kahane and Salem, see \cite[p. 170, Proposition 4]{KaSa94}.

 \begin{theorem} \label{thm5.2}
Let $\varphi \in A(\T)$ and $S$ be a pseudo-measure on $\T$. If $\varphi \cdot S = 0$ then $\varphi$ vanishes on the support of $S$.
 \end{theorem}

We will also use the following result due to Helson, see \cite[pp. 199-200]{Hel10}.

 \begin{theorem} \label{thm5.3}
Let $g$ be a function on $\Z$ which attains only finitely many different values. If the distribution $\ft{g}$ vanishes on some open interval $(a,b)$ on the circle $\T$, then $g$ is a periodic function on $\Z$.
 \end{theorem}

 Theorems \ref{thm5.2} and \ref{thm5.3} can be used  in the role played above by Theorems \ref{thm4.1}  and \ref{thm5.1}, to prove that any tiling of $\Z$ must be periodic.

\begin{proof}[Proof of \thmref{thm6.1}]

Assume that $f+\Lambda$ is a tiling of $\Z$ at some level $w$, where $f\in\ell^1(\mathbb Z)$
and $\Lambda \subset \mathbb Z$. Hence the condition \eqref{5.2b} is satisfied.
Let $\varphi$ be a smooth function on $\mathbb T$ such that $\varphi(0) = 0$. From \eqref{5.2b} we get
$$
\varphi \cdot \ft f \cdot \ft \1_\Lambda = 0.
$$
Since $\varphi \cdot \ft f \in A(\T)$, it follows from \thmref{thm5.2} that $\supp (\ft\1_\Lambda) \subset \{\varphi \cdot \ft f = 0\}.$ As this holds for any
smooth $\varphi$ vanishing at $0$, this shows that
\begin{equation}\label{5.4}
\supp (\ft\1_\Lambda) \subset \{\ft f = 0\} \cup \{0\}.
\end{equation}

Since $f$ is not identically zero,  this implies (as before) that there is an open interval $(a,b)$ disjoint from $\supp(\hat \1_\Lambda)$. Hence $\1_\Lambda$ is a function on $\Z$ which attains only finitely many different values (namely, the values $0$ and $1$), and which has a spectral gap. By \thmref{thm5.3} the function $\1_\Lam$ thus must be periodic, hence $\Lambda$ is a periodic set.
\end{proof}


\section{Open problems}
\label{sec:open}

\subsection{} 

In the proof of \thmref{thm1.1} we constructed a positive function $f\in L^1(\mathbb R)$ and a discrete set 
$\Lambda$ of bounded density, such that $f+\Lambda$ is a tiling at level $1$, but $\Lambda$ is not a finite union of periodic sets.
The fact that $f$ tiles with the translation set $\Lambda$ was deduced from the property that
$\hat \delta_\Lambda=\delta_0$ in some interval $(-a,a)$, and $\hat f$ is supported in $(-a,a).$

Notice that the function $f$ in this example also admits a periodic tiling, namely, it also tiles with 
the translation set $\mathbb Z$ (at the same level). This follows from the same considerations, 
as we also have $\hat \delta_{\mathbb{Z}}=\delta_0$ in the same interval $(-a,a)$.

It seems an interesting question whether this phenomenon holds in general. That is, if a function $f$ tiles $\R$ with a given translation set $\Lambda,$ does it necessarily tile also with some other translation set $\Lambda'$ which is a finite union of periodic sets (at the same level, or at some other level)?
This may be seen as a version of the ``periodic tiling conjecture'' \cite{grunbaum1986tilings,lagarias1997spectral} for functions.

\subsection{} 
Another interesting question is whether \thmref{thm1.1} may be strengthened by taking $f$ to be an indicator function.
Namely, does there exist a set $\Omega\subset\mathbb R$ of positive and finite Lebesgue measure, which tiles by a non-periodic translation set $\Lambda$ (at level 1)? Remark that such $\Omega$ must be an unbounded set, see \cite[Theorem 6.1]{KoLa96}.



\begin{thebibliography}{99}

\bibitem{grunbaum1986tilings}
B. Gr\"{u}nbaum, G. C. Shephard,
Tilings and patterns.
W. H. Freeman \& Co., New York, NY, USA, 1986.

\bibitem{HaJo94}
V. Havin, B. J\"{o}ricke, 
The uncertainty principle in harmonic analysis. 
Ergebnisse der Mathematik und ihrer Grenzgebiete (3)
[Results in Mathematics and Related Areas (3)], 28.
Springer-Verlag, Berlin, 1994.

\bibitem{Hel10}
H. Helson, 
Harmonic analysis. Second edition. 
Hindustan Book Agency, New Delhi, 2010.

\bibitem{IoKo13}
A. Iosevich, M. N. Kolountzakis, 
Periodicity of the spectrum in dimension one. 
Anal. PDE \textbf{6} (2013), 819--827. 

\bibitem{KaSa94}
J.-P. Kahane, R. Salem,
Ensembles parfaits et s\'{e}ries trigonom\'{e}triques (French).
Second edition. 
Hermann, Paris, 1994.

\bibitem{Kar82}
P. P. Kargaev, 
The Fourier transform of the characteristic function of a set vanishing on an interval (Russian).
Mat. Sb. (N.S.) \textbf{117} (1982), 397--411.
English translation in Math. USSR-Sb. \textbf{45} (1983), 397--410.

\bibitem{KaVo92}
P. P. Kargaev, A. L. Volʹberg, 
Three results concerning the support of functions and their Fourier transforms.
Indiana Univ. Math. J. \textbf{41} (1992), 1143--1164.

\bibitem{Ko00a}
M. N. Kolountzakis, 
Non-symmetric convex domains have no basis of exponentials.
Illinois J. Math. \textbf{44} (2000), 542--550.

\bibitem{Ko00b}
M. N. Kolountzakis, 
Packing, tiling, orthogonality and completeness. 
Bull. London Math. Soc. \textbf{32} (2000), 589--599. 

\bibitem{KoLa96}
M. N. Kolountzakis, J. C. Lagarias, 
Structure of tilings of the line by a function. 
Duke Math. J. \textbf{82} (1996), 653--678. 

\bibitem{LaWa96}
J. C. Lagarias, Y. Wang,
Tiling the line with translates of one tile.
Invent. Math. \textbf{124} (1996), 341--365.

\bibitem{lagarias1997spectral}
J. C. Lagarias, Y. Wang,
Spectral Sets and Factorizations of Finite Abelian Groups.
Journal of Functional Analysis, \textbf{145} (1997), 73--98.

\bibitem{LepHo91}
H. Leptin, D. M\"{u}ller,
Uniform partitions of unity on locally compact groups. 
Adv. Math. \textbf{90} (1991), 1--14.

\bibitem{New77}
D. J. Newman, 
Tesselation of integers. 
J. Number Theory \textbf{9} (1977), 107--111. 



\end{thebibliography}
\end{document}